\documentclass{amsart}
\makeatletter

\usepackage{import}
\usepackage[T1]{fontenc}
\usepackage{textcmds}  
\usepackage{amsmath, amssymb, amsfonts, amstext, verbatim, amsthm, mathrsfs}
\numberwithin{equation}{section} 
\usepackage{geometry}
\usepackage{mathtools}
\usepackage[mathcal]{eucal}
\usepackage{microtype}
\usepackage{thmtools, thm-restate}
\usepackage[all]{xy}
\usepackage{tikz-cd}
\usepackage{tikz} 
\usepackage{enumerate}
\usepackage{enumitem}
\usepackage[modulo]{lineno}
\usepackage[colorlinks=true,linkcolor=blue,citecolor=blue,urlcolor=blue,citebordercolor={0 0 1},urlbordercolor={0 0 1},linkbordercolor={0 0 1}]{hyperref} 
\usepackage[shortalphabetic]{amsrefs} 
\usepackage[nameinlink]{cleveref}

\def\makeCal#1{%
\expandafter\newcommand\csname c#1\endcsname{\mathcal{#1}}}
\def\makeBB#1{%
\expandafter\newcommand\csname b#1\endcsname{\mathbb{#1}}}
\def\makeFrak#1{%
\expandafter\newcommand\csname f#1\endcsname{\mathfrak{#1}}}
\def\makeScr#1{%
\expandafter\newcommand\csname s#1\endcsname{\mathscr{#1}}}

\count@=0
\loop
\advance\count@ 1
\edef\y{\@Alph\count@}%
\expandafter\makeCal\y
\expandafter\makeBB\y
\expandafter\makeFrak\y
\ifnum\count@<26
\repeat

\newcommand{\Alg}{\mathbf{Alg}}

\newcommand{\CDGA}{\mathbf{CDGA}}

\DeclareMathOperator*{\colim}{co{\lim}~}

\newcommand{\Com}{\mathbf{Com}}

\newcommand{\Gp}{\mathbf{Gp}}

\DeclareMathOperator{\hcolim}{hocolim~}

\DeclareMathOperator{\Ho}{Ho}

\newcommand{\Sch}{\mathbf{Sch}}
\newcommand{\Set}{\mathbf{Set}}

\DeclareMathOperator{\Spec}{Spec~}

\newcommand{\Top}{\mathbf{Top}}

\newcommand{\Vect}{\mathbf{Vect}}

\theoremstyle{plain}
\newtheorem{thm}{Theorem}[section]
\newtheorem{cor}{Corollary}[thm]
\newtheorem{lem}[thm]{Lemma}

\newtheorem{prop}[thm]{Proposition}

\theoremstyle{definition}
\newtheorem{defn}{Definition}[section]

\newtheorem{eg}[thm]{Example}

\theoremstyle{remark}
\newtheorem{rmk}{Remark}

\AtEndDocument{%
\bibliographystyle{alpha}
\nocite{*}
\bibliography{./reference.bib}
}

\title{Commuting Schemes of Upper Triangular Matrices and Representation Homology}
\author{Guanyu Li}
\date{\today}

\begin{document}


\begin{abstract}
We establish a connection between generalised commuting schemes $C_g(U_n)$ of higher genus $g$, which are associated with a group scheme $U_n$ consisting of upper triangular unipotent matrices, and the representation homology $HR_*(\Sigma_g,U_n)$ of a Riemann surface $\Sigma_g$ with coefficients in the group $U_n$.
As an outcome, we provide a numerical criterion for determining whether commuting schemes $C(U_n)$ form a complete intersection.
\end{abstract}

\maketitle

\section{Introduction}

The commuting pairs of matrices form an affine scheme which is called the commuting scheme.
These schemes have been extensively studied in relation to different areas of mathematics.
In particular, much work has been done on the geometric properties of commuting schemes.

Motzkin and Taussky \cite{MT55} and Gerstenhaber \cite{G61} proved that for any $n\in\mathbb{N}$ the classical variety of all the pairs $(A,B)$ of $n\times n$ matrices over $k$ such that $[A,B] = 0$ is irreducible.
In 1979, Richardson \cite{R69} proved that for a reductive group $G$ (resp. a reductive Lie algebra $\mathfrak{g}$) over an algebraically closed field of characteristic $0$, the commuting scheme of $G$ (resp. $\mathfrak{g}$) is irreducible.
This result generalises from the classical variety of all pairs of matrices to the variety of pairs of matrices of certain types.
Later, Baranovsky \cite{B01} and Basili \cite{Bas03} established the irreducibility of the commuting scheme of $n\times n$ nilpotent matrices, under mild conditions on the characteristic of the field $k$.

Properties other than irreducibility have also been studied.
For instance, in \cite{Bas17}, Basili gave some properties of being a complete intersection of commuting pairs.
Knutson \cite{K05} shows that the commuting variety of $\mathfrak{gl}_n$ is an irreducible component of a complete intersection called the diagonal commutator scheme.
Keeton \cite{K96} studied the question when a commuting scheme is a complete intersections.

The commuting scheme associated to a group scheme is a special case of the so-called representation scheme $\mathrm{Rep}_G(\Gamma)$, which is a scheme parametrizing all representations of a group $\Gamma$ in $G$.
While this is a useful generalisation, representation schemes do not naturally behave well. First, the schemes are generally very singular, and hence one needs to resolve the singularities of $\mathrm{Rep}_G(\Gamma)$. Second, in practice, the group $\Gamma$ is generally to be taken the fundamental group of a space such as (compact orientable) surfaces, hyperbolic 3-manifolds and knot complements in $S^3$, where all the spaces mentioned here are aspherical.
For more general spaces, however, one needs to take into account a higher homotopy information of the space.

Derived algebraic geometry offers a possible solution to remedy the deficiencies of representation schemes. A derived version of the representation schemes, representation homology, was introduced in \cite{BRY22}.
Notably, earlier versions of representation homology were developed for associative algebras and Lie algebras, introduced respectively in \cite{BKR13} and \cite{BFPRW17}.

Associated to a topological space $X$ and a group scheme $G$, the derived representation scheme, defined in the most approachable way, is the \emph{non-abelian} derived functor (in the sense of \cite{Q67}) of the representation scheme, and representation homology is the homotopy group of its simplicial coordinate ring.
In particular, the $0$-th homotopy group gives the classical coordinate ring of $\mathrm{Rep}_G(\pi_1(X))$.
Behaving like abelian derived functors such as $\mathrm{Tor}$, representation homology is independent of the choice of resolutions, though the existence and calculations rely on resolutions heavily.

Although there are many other equivalent definitions, for instance the derived local system (see \cite{K01}) and the mapping stack interpretation (see \cite{TV08} and \cite{BRY19}*{Appendix}), the construction in \cite{BRY22} is actually computable.
This terminology is motivated by the fact (see \cite{BRY22}) that $ {HR}_\ast(X,G) $ admits an elementary definition parallel to the classical definition of ordinary homology $ H_*(X,A)$ of $X$ with coefficients in an abelian group $A$ and the Loday-Pirashvili definition of higher Hochschild homology $ HH_*(X,R)$ with coefficients in a commutative ring $R$.

It is generally accepted that derived algebraic geometry is powerful, and a derived object contains enormous information.
However, there seems to be a gap between derived algebraic geometry and classical algebraic geometry: relatively few examples demonstrate how the computations of derived objects give geometric data of classical algebro-geometric objects.
One such case relates the vanishing of the representation homology of noncommutative complete intersection associative algebra $A$ to the classical representation scheme $\mathrm{Rep}_n(A)$ being complete intersections (\cite{BFR14}*{Theorem 24}).

The main result of this paper relates the properties of representation homology for groups to the classical geometric properties of commuting schemes. We have

\begin{thm}[\Cref{Proposition_UnipotnetGp}, \Cref{Thm_List} and \Cref{Proposition_U6NonCI}]\label{Thm_Main}
Let $U_n$ be the algebraic group of $n\times n$ upper triangular unipotent matrices over a characteristic $0$ field $k$, and let $C_g(U_n)$ be its commuting scheme of genus $g$ (see \Cref{Def_CommSchHigherGenus}).
Then $C_g(U_n)$ is a (global) complete intersection if and only if the representation homology vanishes in degrees greater or equal than the size of matrices, i.e.
\begin{equation}\label{Eq_Vanishing}
HR_i(\Sigma_g,U_n)=0\;\;\;\forall i\geq n,
\end{equation}
where $\Sigma_g$ is the Riemann surface of genus $g$.
In the $g=1$ case, condition \ref{Eq_Vanishing} holds for $n\leq5$, so that $C(U_2),C(U_3),C(U_4),C(U_5)$ are complete intersections.
In contrast, $C(U_n)$ is not a complete intersection for $n\geq6$.
\end{thm}

In the literature, people have done calculations of representation homology with reductive group coefficients, for example \cite{BRY19}*{Conjecture 1.3}.
To our knowledge, there are very few results with unipotent coefficients.
As \Cref{Thm_Main} indicates, they are very different from the other known cases.

\subsection{Organisation of the paper}

The paper consists of two main sections :
In \Cref{Sec_Pre}, we collect necessary material about commuting schemes, and tools to analyse their geometry.
We then briefly summarise the construction and basic properties of representation homology, and set up background for computations.
In \Cref{Sec_Main}, we present our main results in \Cref{Thm_List}, and the main technical tools in \Cref{Proposition_UnipotnetGp} and \Cref{Proposition_BorelGp} with their proofs.
In particular, with the help of \verb|Macaulay2| (\cite{M2}), we can compute the representation homology $HR_*(\Sigma_g,U_n)$ for small $n$, implying that for natural numbers $n\leq5$, $C(U_n)$ is a complete intersection.
One might expect that this list of complete intersections goes further, however it is not true.
As a counterexample, we prove \Cref{Proposition_U6NonCI}, which shows that $C(U_6)$ is not a complete intersection.

\subsection{Notation and conventions}

Throughout this paper, $k$ will denote a field of characteristic $0$ and $\Vect_k$ is the category of $k$-vector spaces. Unless specified, the tensor product $\otimes$ will be over $k$.

By a $k$-algebra $R$, we mean a \emph{commutative} unital ring $R$ over $k$. We shall denote by $\Alg_k$ the category of $k$-algebras. We will always use homological convention, namely the differentials are always of degree $-1$. A DG algebra is a positive graded \emph{commutative} algebra with compatible differential maps, i.e.
\begin{displaymath}
xy=(-1)^{\deg x\deg y}yx
\end{displaymath}
and
\begin{equation}\label{Eq_GradedLeibnitz}
d(xy)=d(x)y+(-1)^{\deg x}xd(y)
\end{equation}
for homogenerous elements $x,y$.
The category of all such DG algebras will be denoted by $\CDGA_k$.

We denote by $\Top$ and $\Top_*$ the categories of compactly generated weakly Hausdorff spaces and pointed spaces respectively, by $\Gp$ the category of groups, and $\Sch_k$ the category of schemes over $\Spec k$.

Given any category $\cC$, denote by $s\cC$ the category of simplicial objects of $\cC$, i.e. $s\cC:=\mathrm{Funct}(\bf\Delta^{\mathrm{op}},\cC)$.
For results about simplicial sets and model categories, we refer to \cite{GJ09}.

We shall not distinguish between a simplicial set $X$ and its geometric realization $|X|$ (see \Cref{Eq_Simplicial}).

\subsection{Acknowledgment}
I would like to thank my advisor Yuri Berest first, who proposed the problem, and ask me many valuable questions which help improve my understanding.
I would also like to thank Mike Stillman, who helped me 
with building up the \verb|Macaulay2| packages, and analysing the pathology that came along the computations.

\section{Preliminaries}\label{Sec_Pre}

In this section, we first recall some results about Koszul complexes and regular sequences.
Then we recall definitions and properties of affine group schemes and their commuting schemes.
After that we discuss briefly about the representation scheme and its derived version - representation homology, which is first introduced in \cite{BKR13}.

\subsection{Koszul complex and codimension}

All the definitions and results presented in this section are standard (for instance, see \cites{Eisenbud95,Matsumura89}), however, we think that it would still be helpful to gather them together here for later use in the paper.

\begin{defn}
Given a commutative ring $R$ and elements $f_1,\cdots,f_r$ in $R$, the {\bf Koszul complex} (or {\bf Koszul algebra}) is the (commutative) DG algebra defined by
\begin{enumerate}
\item the underlying graded algebra is
\begin{displaymath}
\bigwedge_R\left(\bigoplus_{i=1}^rR\cdot t_i\right)
\end{displaymath}
where $t_i$ is of (homological) degree $1$ for all $i$,
\item the differential satisfies
\begin{displaymath}
d(t_i)=f_i
\end{displaymath}
and is extended by the graded Leibniz rule \Cref{Eq_GradedLeibnitz}.
\end{enumerate}
We denote by $\mathrm{Kos}^R_*(f)$ the Koszul complex of $f:=\{f_1,\cdots,f_r\}$. When the ambient ring $R$ is clear, we simply write $\mathrm{Kos}_*(f)$.
\end{defn}

Since any $\mathrm{Kos}^R_*(f)$ is naturally a DG algebra, its homology is a graded algebra.

\begin{lem}\label{Lemma_KoszulZeroSubseq}
Let $f=\{f_1,\cdots,f_r\}$ be a sequence in a $k$-algebra $R$, let $g=\{g_1=0,\cdots,g_t=0\}$ be the zero sequence. Then
\begin{enumerate}
\item $\mathrm{Kos}_*(g)\cong R\otimes\Lambda_k[\xi_1,\cdots,\xi_t]\cong \Lambda_R[\xi_1,\cdots,\xi_t]$ as vector spaces, and the differentials are $0$.
\item
\begin{displaymath}
\mathrm{Kos}_*(f\cup g)=\mathrm{Kos}_*(f)\otimes_R\mathrm{Kos}_*(g).
\end{displaymath}
\item
\begin{displaymath}
H_*(\mathrm{Kos}(f\cup g))\cong H_*(\mathrm{Kos}(f))\otimes\Lambda_k[\xi_1,\cdots,\xi_t].
\end{displaymath}
\end{enumerate}
\end{lem}
\begin{proof}
The first two are immediate.
By K\"unneth formula (\cite{Weibel94}*{Theorem 3.6.3}) there is a short exact sequence
\begin{displaymath}
0\to\bigoplus_{p+q=n}H_p(P_\bullet)\otimes_RH_q(Q_\bullet)\to H_n((P\otimes_R Q)_\bullet)\to\bigoplus_{p+q=n-1}\mathrm{Tor}_1^R(H_p(P_\bullet),H_q(Q_\bullet))\to0
\end{displaymath}
where $P_\bullet$ is taken $\mathrm{Kos}_*(g)$ and $Q_\bullet$ is taken $\mathrm{Kos}_*(f)$. By Part (1), $\mathrm{Tor}_1^R(H_p(P_\bullet),H_q(Q_\bullet))$ is always $0$, hence we have the desired isomorphism.
\end{proof}

Let's then recall the notion of regular sequences:

\begin{defn}
Let $R$ be a commutative ring.
A sequence of elements $\{f_1,\cdots,f_r\}\subseteq R$ is called a \textbf{regular sequence} if
\begin{itemize}
\item $(f_1,\cdots,f_r)\neq R$, and
\item For $i=1,\cdots,r$, $f_i$ is a nonzerodivisor on $\frac{R}{(f_1,\cdots,f_{i-1})}$.
\end{itemize}
\end{defn}

And Koszul complex is highly related to the notion of regular sequence:

\begin{thm}\label{Thm_CA}
Given homogeneous elements $f_1,\cdots,f_r\in k[x_1,\cdots,x_n]$ ($\deg x_i$ is arbitrary) of positive degrees, denote the sequence by $f=\{f_1,\cdots,f_r\}$, then the following are equivalent:
\begin{enumerate}
\item $\{f_1,\cdots,f_r\}$ form a regular sequence.
\item $\mathrm{codim}~(f_1,\cdots,f_j)=j$ for $1\leq j\leq r$.
\item $\mathrm{codim}~(f_1,\cdots,f_r)=r$.
\item $H_j(\mathrm{Kos}_*(f))=0$ for $1\leq j\leq r$.
\item $H_1(\mathrm{Kos}_*(f))=0$.
\end{enumerate}
\end{thm}

\Cref{Thm_CA} combines several general results.
The proofs can be found in \cite{Eisenbud95}*{Theorem 10.2}, \cite{Eisenbud95}*{Corollary 17.5}, \cite{Eisenbud95}*{Proposition 18.2}, \cite{Matsumura89}*{Theorem 16.5, part ii} and \cite{Matsumura89}*{Theorem 17.4}.

\begin{cor}\label{Cor_Subseq}
Within the same setting, any subsequence of $f$ with any order is a regular sequence.
\end{cor}

In terms of geometry, we also have the notion

\begin{defn}
We say that an algebra $R$ of finite type over $k$ is a \textbf{global complete intersection} if there exists \emph{a} presentation $\displaystyle R=\frac{k[x_1,\cdots,x_n]}{(f_1,\cdots,f_r)}$ such that $\dim(R)=n-r$.
\end{defn}

As an immediate consequence of \Cref{Thm_CA}, we have the following corollary, which is true basically by the properties of polynomial algebras:

\begin{cor}\label{Cor_CI}
Let $X$ be an affine subscheme of $\mathbb{A}^n_k$ of dimension $n-r$. Suppose there is a presentation $\displaystyle\Spec\frac{k[x_1,\cdots,x_n]}{(f_1,\cdots,f_r)}$ where $f_j$ are homogeneous of positive degrees (again $\deg x_i$ is arbitrary), then $X$ is a complete intersection if and only if $\{f_1,\cdots,f_r\}$ is a regular sequence.
\end{cor}

\subsection{Group schemes and the grading}

\begin{defn}
An {\bf affine group scheme} $G$ is a group object in the category $\Sch_k$ which is affine as a scheme. We shall denote by $\mathscr{O}(G)$ the coordinate ring of $G$.
\end{defn}

In this paper, we will use the functor of points interpretation of a group scheme, which says a scheme $G\in\Sch_k$ is a group scheme if and only if by viewing it a functor ${\bf Alg}_k\to\Set$, it has a natural factorisation
\begin{displaymath}
{\bf Alg}_k\to\Gp\xrightarrow{U}\Set
\end{displaymath}
where $U$ is the forgetful functor $\Gp\to\Set$. This is equivalent to saying that for each $A\in\Alg_k$, there is a functorial way of assigning the set $G(A)$ a group structure.

\begin{eg}
The basic example is the general linear group $GL_n$, the affine (group) scheme such that for any $k$-algebra $R$,
\begin{displaymath}
GL_n(R):=\left\{M=\begin{bmatrix}a_{1,1}&a_{1,2}&\cdots&a_{1,n}\\
a_{2,1}&a_{2,2}&\cdots&a_{2,n}\\
\vdots&\vdots&\ddots&\vdots\\ a_{n,1}&a_{n,2}&\cdots&a_{n,n}\end{bmatrix}\middle|a_{i,j}\in R,\;\;\det M\in R^\times\right\}.
\end{displaymath}
Its coordinate ring is $k[x_{i,j},\det^{-1}]_{1\leq i,j\leq n}$, with comultiplication
\begin{equation}\label{Eq_Comultiplication}
\Delta(x_{i,j})=\sum_{t=1}^nx_{i,t}\otimes x_{t,j}.
\end{equation}
\end{eg}

We will concentrate on the following two types of group schemes.

\begin{defn}
Let $U_n$ be the affine (group) scheme such that for any $k$-algebra $R$,
\begin{displaymath}
U_n(R):=\left\{\begin{bmatrix}1&a_{1,2}&a_{1,3}&\cdots&a_{1,n}\\
0&1&a_{2,3}&\cdots&a_{2,n}\\
0&0&1&\cdots&a_{3,n}\\
0&0&0&\ddots&\vdots\\ 0&0&0&0&1\end{bmatrix}\middle|a_{i,j}\in R\right\}.
\end{displaymath}
The coordinate ring of $U_n$ is the polynomial algebra $\displaystyle k[x_{i,j}]_{1\leq i<j\leq n}$.
We call it the {\bf unipotent $n\times n$-matrix group scheme}.
We call the matrix $X\in U_n(k[x_{i,j}]_{1\leq i<j\leq n})$ of the form
\begin{equation}\label{Eq_UnMatrixGenerator}
X:=\begin{bmatrix}1&x_{1,2}&x_{1,3}&\cdots&x_{1,n}\\
0&1&x_{2,3}&\cdots&x_{2,n}\\
0&0&1&\cdots&x_{3,n}\\
\vdots&\vdots&\vdots&\ddots&\vdots\\ 0&0&0&\cdots&1\end{bmatrix}
\end{equation}
the generic {\bf coordinate matrix of $U_n$} and also write $\mathscr{O}(U_n)=k[X]:=k[x_{i,j}]_{1\leq i<j\leq n}$.
\end{defn}

\begin{defn}
Similarly, let $B_n$ be the affine (group) scheme such that for any $k$-algebra $R$,
\begin{displaymath}
B_n(A):=\left\{\begin{bmatrix}a_{1,1}&a_{1,2}&a_{1,3}&\cdots&a_{1,n}\\
0&a_{2,2}&a_{2,3}&\cdots&a_{2,n}\\
0&0&a_{3,3}&\cdots&a_{3,n}\\
0&0&0&\ddots&\vdots\\ 0&0&0&0&a_{n,n}\end{bmatrix}\middle|a_{i,j}\in R,\;a_{1,1}\cdots a_{n,n}\in R^\times\right\}.
\end{displaymath}
The coordinate ring is $\displaystyle\frac{k[x_{i,j},t]_{1\leq i\leq j\leq n}}{(t\cdot x_{1,1}\cdots x_{n,n}-1)}$. We call it the {\bf Borel $n\times n$-matrix group scheme}. We call the matrix $\displaystyle X\in B_n\left(\frac{k[x_{i,j},t]_{1\leq i\leq j\leq n}}{(t\cdot x_{1,1}\cdots x_{n,n}-1)}\right)$ of the form
\begin{equation}\label{Eq_BnMatrixGenerator}
X:=\begin{bmatrix}x_{1,1}&x_{1,2}&x_{1,3}&\cdots&x_{1,n}\\
0&x_{2,2}&x_{2,3}&\cdots&x_{2,n}\\
0&0&x_{3,3}&\cdots&x_{3,n}\\
\vdots&\vdots&\vdots&\ddots&\vdots\\ 0&0&0&\cdots&x_{n,n}\end{bmatrix}
\end{equation}
the {\bf coordinate matrix of $B_n$} and write $\displaystyle\mathscr{O}(B_n)=k[X]:=\frac{k[x_{i,j},t]_{1\leq i\leq j\leq n}}{(t\cdot x_{1,1}\cdots x_{n,n}-1)}$.
\end{defn}

\begin{eg}
Consider $G=U_3$, given any $k$-algebra $R$ and two generic matrices $A,B\in U_3(R)$,
\begin{displaymath}
A=\begin{pmatrix}1&a_1&a_3\\ &1&a_2\\ &&1\end{pmatrix},B=\begin{pmatrix}1&b_1&b_3\\ &1&b_2\\ &&1\end{pmatrix}
\end{displaymath}
if we give an internal grading on the matrix so that $\deg a_1=\deg a_2=\deg b_1=\deg b_2=1$ and $\deg a_3=b_3=2$, then in the matrix
\begin{displaymath}
AB=\begin{pmatrix}1&a_1+b_1&a_3+a_1b_2+b_3\\ &1&a_2+b_2\\ &&1\end{pmatrix},
\end{displaymath}
one notices that all elements $a_1+b_1,a_2+b_2,a_3+a_1b_2+b_3$ are homogeneous, of degrees $1,1$ and $2$ respectively, i.e. the matrix product preserves the internal grading.\par
There is a similar internal grading on $G=B_2$.
\end{eg}

In general, one could define the internal grading by
\begin{equation}\label{Eq_Grading}
\deg x_{i,j}=j-i,
\end{equation}
which is a constant in the diagonal direction, and preserved by the multiplication. Since there is only positive degrees, one sets
\begin{displaymath}
\deg0=-\infty.
\end{displaymath}

\begin{lem}\label{Lemma_GroupProperties}
\begin{enumerate}

\item The grading (\ref{Eq_Grading}) makes $\mathscr{O}(U_n)$ and $\mathscr{O}(B_n)$ graded Hopf algebras.

\item\begin{enumerate}
\item The kernel of the conunit map $\epsilon_{U_n}:\mathscr{O}(U_n)\to k$ is generated by a regular sequence $\{x_{i,j}\}_{1\leq i<j\leq n}$.
\item The kernel of the conunit map $\epsilon_{B_n}:\mathscr{O}(B_n)\to k$ is generated by a regular sequence $\{x_{i,j}-\delta_j^i\}_{1\leq i\leq j\leq n}$ where $\displaystyle\delta_j^i=\left\{\begin{matrix}1&i=j\\0&\text{otherwise}\end{matrix}\right.$ is the Kronecker delta.
\end{enumerate}
\end{enumerate}
\end{lem}

\begin{proof}
\begin{enumerate}
\item For $U_n$, the coordinate ring is a polynomial ring and the grading is defined on its generators. So $\mathscr{O}(U_n)$ is a graded algebra. It is easy to see from \Cref{Eq_Comultiplication} and \Cref{Eq_Grading} that the grading is compatible with the comultiplication since the comultiplication is inherited from $GL_n$. Furthermore, any matrix in $U_n(R)$ has a decomposition
\begin{displaymath}
A=I-N
\end{displaymath}
where $I$ is the identity matrix and $N$ is a nilpotent matrix ($N^{n-1}=0$), hence
\begin{displaymath}
A^{-1}=I+N+\cdots+N^{n-2},
\end{displaymath}
i.e. the antipodal map could be interpreted by comultiplications. Concretely,
\begin{displaymath}
S(x_{i,j})=\sum_{k=0}^{n-2}\sum_{1\leq i<t_1<\cdots<t_k<j\leq n-1}(-1)^{k+1}x_{i,t_1}x_{t_1,t_2}\cdots x_{t_k,j},
\end{displaymath}
hence $S$ preserves the grading.\par
For $B_n$, we see by \Cref{Eq_Grading} $\deg x_{i,i}=0$ and we set in addition $\deg t=0$, hence $\mathscr{O}(B_n)$ is a quotient algebra from a polynomial algebra by a homogeneous equation $t\cdot x_{1,1}\cdots x_{n,n}-1$. Thus $\mathscr{O}(B_n)$ is a graded algebra. The grading is compatible with the comultiplication because of the same reason, and so is the antipodal map (or concretely this time $\displaystyle S(x_{i,j})=x_{j,j}^{-1}\sum_{k=0}^{n-2}\sum_{1\leq i<t_1<\cdots<t_k<j\leq n-1}(-1)^{k+1}\frac{x_{i,t_1}}{x_{i,i}}\frac{x_{t_1,t_2}}{x_{t_1,t_1}}\cdots\frac{x_{t_k,j}}{x_{t_k,t_k}}$).

\item Since $\mathscr{O}(U_n)$ is the polynomial ring generated by the elements $\{x_{i,j}\}_{1\leq i<j\leq n}$, the statement is clear for $\mathscr{O}(U_n)$.\par
Notice that
\begin{align*}
\mathscr{O}(B_n)&=\frac{k[x_{i,j},t]_{1\leq i\leq j\leq n}}{(t\cdot x_{1,1}\cdots x_{n,n}-1)}\\
&\cong\frac{k[x_{i,i},t]_{1\leq i\leq n}}{(t\cdot x_{1,1}\cdots x_{n,n}-1)}[x_{i,j}]_{1\leq i<j\leq n}\\
&\cong\left(\bigotimes_{1\leq i\leq n}k[x_{i,i},x_{i,i}^{-1}]\right)[x_{i,j}]_{1\leq i<j\leq n}
\end{align*}
where the tensor product is taken over $k$. Hence, the statement is also clear for $\mathscr{O}(B_n)$.
\end{enumerate}
\end{proof}

\subsection{Commuting schemes and higher genus generalisations}

Let $G$ be an affine group scheme over $k$ with coordinate ring $\mathscr{O}(G)$.

Denote by $C(G)$ the {\bf commuting scheme} of $G$, and in terms of functor of points, it is the affine scheme defined by
\begin{align*}
C(G):{\bf Alg}_k&\to\Set\\
R&\mapsto\{(a,b)\in G(R)\times G(R)\mid ab=ba\},
\end{align*}
where with the notation $[a,b]:=aba^{-1}b^{-1}$, one could also write the set
\begin{displaymath}
C(G)(R)=\{(a,b)\in G(R)\times G(R)\mid[a,b]=1\}.
\end{displaymath}
The scheme $C(G)$ is actually affine for the following reason: it is the closed subscheme of $G\times G$ cut out by the equation
\begin{displaymath}
xyx^{-1}y^{-1}=e,
\end{displaymath}
where this equation is an algebraic equation since the group multiplication and inverse are algebraic in the group scheme $G$. Its coordinate ring will be in \Cref{Eq_GeneralCommutingCoordinateRing}.

\begin{rmk}
Although historically we call the object a commuting variety, it really should be called a commuting scheme without irreducibility and reducedness.
\end{rmk}

\begin{eg}\label{Eg_MatrixGenerator}
The group scheme $GL_n$ is a matrix group scheme, in the sense that there is a matrix $X$ of symbols (variables over $k$ or elements in $k$), satisfying
\begin{enumerate}
\item For any $k$-algebra $R$, $GL_n(R)$ is the set of matrices by evaluating $X$ in $R$, and
\item the multiplication in $GL_n(R)$ is exactly the matrix multiplication of evaluations of $X$.
\end{enumerate}
We say an affine group scheme $G$ has a matrix generator, if $G$ is a closed subgroup of $GL_n$. Then the inclusion $i:G\hookrightarrow GL_n$ gives a matrix generator $i^*X$ of $G$ in terms of the conditions just above:
\begin{enumerate}
\item For any $k$-algebra $R$, $G(R)$ is the set of matrices by evaluating $i^*X$ in $R$, and
\item the multiplication in $G(R)$ is exactly the matrix multiplication of evaluations of $i^*X$.
\end{enumerate}
In this case, we shall denote by $k[X]$ the coordinate ring of $G$.\par
Now if $G$ is a group scheme with a matrix generator, then the coordinate ring of $C(G)$ is
\begin{displaymath}
\frac{k[X,Y]}{(XYX^{-1}Y^{-1}-I)},
\end{displaymath}
where $k[X,Y]=k[X]\otimes_kk[Y]$ is the coordinate ring of $G\times G$, and $XYX^{-1}Y^{-1}=I$ is the set of $n\times n$ equations, claiming the entries in the matrix $XYX^{-1}Y^{-1}$ are equal to the corresponding entries in $I$.\par
We shall give example to illustrate this:
\begin{enumerate}[label=(\roman*)]
\item Let $G=U_3$, then the matrix generator is $X=\begin{pmatrix}1&x_{1,2}&x_{1,3}\\ 0&1&x_{2,3}\\ 0&0&1\end{pmatrix}$, and
\begin{displaymath}
XYX^{-1}Y^{-1}=\begin{pmatrix}1&0&x_{1,2}y_{2,3}-y_{1,2}x_{2,3}\\ 0&1&0\\ 0&0&1\end{pmatrix},
\end{displaymath}
thus the coordinate ring of $C(U_3)$ is
\begin{displaymath}
\frac{k[X,Y]}{(XYX^{-1}Y^{-1}-I)}=\frac{k[x_{1,2},x_{1,3},x_{2,3},y_{1,2},y_{1,3},y_{2,3}]}{(x_{1,2}y_{2,3}-y_{1,2}x_{2,3})}.
\end{displaymath}
\item Let $G=B_2$, then the matrix generator is $X=\begin{pmatrix}x_{1,1}&x_{1,2}\\ 0&x_{2,2}\end{pmatrix}$, and
\begin{displaymath}
XYX^{-1}Y^{-1}=\begin{pmatrix}1&-x_{1,1}x_{1,2}y_{1,1}^2st+x_{1,1}^2y_{1,1}y_{1,2}st+x_{1,1}x_{1,2}s-y_{1,1}y_{1,2}t\\ 0&1\end{pmatrix},
\end{displaymath}
where $s=(\det X)^{-1}=(x_{1,1}x_{2,2})^{-1}$ and $t=(\det Y)^{-1}=(y_{1,1}y_{2,2})^{-1}$, thus the coordinate ring of $C(B_2)$ is
\begin{displaymath}
\frac{k[X,Y]}{(XYX^{-1}Y^{-1}-I)}=\frac{k[x_{1,1},x_{1,2},x_{2,2},(x_{1,1}x_{2,2})^{-1},y_{1,1},y_{1,2},y_{2,2},(y_{1,1}y_{2,2})^{-1}]}{(-x_{1,1}x_{1,2}y_{1,1}^2st+x_{1,1}^2y_{1,1}y_{1,2}st+x_{1,1}x_{1,2}s-y_{1,1}y_{1,2}t)}.
\end{displaymath}
\end{enumerate}
The coordinate matrix of $U_n$ (resp. $B_n$) is the matrix generator of $U_n$ (resp. $B_n$).
\end{eg}

We conclude this section by generalising the notion of commuting scheme.

\begin{defn}\label{Def_CommSchHigherGenus}
Given an affine group scheme $G$ over $k$ with coordinate ring $\mathscr{O}(G)$ and a positive integer $g\geq1$, {\bf the commuting scheme of genus $g$} is the scheme
\begin{align*}
C_g(G):{\bf Alg}_k&\to\Set\\
R&\mapsto\left\{(a_1,b_1,\cdots,a_g,b_g)\in\prod_{i=1}^{2g} G(R)\mid[a_1,b_1][a_2,b_2]\cdots[a_g,b_g]=1\right\}.
\end{align*}
\end{defn}

Notice that when $g=1$, the higher genus commuting scheme goes back to the classical commuting scheme. The reason why we give it such a name will be in \Cref{Eg_DerivingCommutingVarOfHigherGenus}.

\subsection{Representation schemes}

Given an affine group scheme $G$ over $k$, by viewing it a functor, one has an adjunction
\begin{equation}\label{Eq_BasicAdj}
(-)_G:\Gp\leftrightarrows\Alg_k:G
\end{equation}
where the left functor $(-)_G:\Gp\to\Alg_k$ exists since a group homomorphism only satisfies algebraic equations. To any group $\Gamma$, the scheme $\mathrm{Rep}_G(\Gamma):=\Spec(\Gamma)_G$ associated to the ring $(\Gamma)_G$ is called the \textbf{representation scheme}.

\begin{eg}\label{Eg_RepVars}
By definition, for any group $\Gamma$ and any group scheme $G$,
\begin{displaymath}
\mathrm{Hom}_{\Alg_k}((\Gamma)_G,R)\cong\mathrm{Hom}_{\Gp}(\Gamma,G(R)).
\end{displaymath}
\begin{enumerate}
\item Consider when $\Gamma=1$ the trivial group, then $\mathrm{Hom}_{\Gp}(\Gamma,G(R))$ consists of only one element. Thus $(\Gamma)_G=k$ since $\mathrm{Hom}_{\Alg_k}((\Gamma)_G,R)$ consists of only one element for any $k$-algebra $R$.
\item Consider when $\Gamma=\mathbb{Z}$ the trivial group, then $\mathrm{Hom}_{\Gp}(\Gamma,G(R))\cong G(R)$ and in order to have $\mathrm{Hom}_{\Alg_k}((\Gamma)_G,R)\cong G(R)$ for all $R$, by definition $(\Gamma)_G\cong\mathscr{O}(G)$.
\item Consider when $\Gamma=\mathbb{Z}\times\mathbb{Z}$, then
\begin{displaymath}
\mathrm{Hom}_{\Gp}(\mathbb{Z}\times\mathbb{Z},G(R))=\{(a,b)\in G(R)\times G(R)\mid ab=ba\}.
\end{displaymath}
This means $(\mathbb{Z}\times\mathbb{Z})_G$ is the coordinate ring of $C(G)$. Hence, the representation scheme is a (wide) generalisation of the commuting scheme.
\end{enumerate}
\end{eg}

\begin{eg}
Suppose the group $\Gamma$ is finitely presented, namely it has a presentation $\Gamma\cong\langle x_1,\cdots,x_m\mid r_1,\cdots,r_n\rangle$, and suppose that the group scheme $G$ has a matrix generator $X$ (\Cref{Eg_MatrixGenerator}), then
\begin{equation}\label{Eq_RepVarCoordinateRing}
\Gamma_G\cong\frac{k[X_1,\cdots,X_m]}{(r_1-1,\cdots,r_n-1)},
\end{equation}
where $X_i$ are the matrix generators (so $k[X_1,\cdots,X_m]$ is the coordinate ring of $m$-copies of $G$), and each $r_j-1$ is the set of $n\times n$-expressions in $k[X_1,\cdots,X_m]$, claiming the entries in the matrix $r_j(X_1,\cdots,X_m)$ are equal to the corresponding entries in $I$, where the matrix $r_j(X_1,\cdots,X_m)$ is the expression by replacing $x_i$ in $r_j$ by the coefficient matrix $X_i$.\par
For instance, take $G=GL_2$ and $\Gamma=\mathbb{Z}\times\mathbb{Z}=\langle x,y\mid xyx^{-1}y^{-1}\rangle$, then by \Cref{Eq_RepVarCoordinateRing}
\begin{displaymath}
(\mathbb{Z}\times\mathbb{Z})\cong\frac{k[X,Y]}{(XYX^{-1}Y^{-1}-I)},
\end{displaymath}
which is the coordinate ring of $C(GL_2)$ as stated in \Cref{Eg_MatrixGenerator}.
\end{eg}


\subsection{Representation homology}

Next we extend the adjunction (\ref{Eq_BasicAdj}) to simplicially, getting another adjunction
\begin{equation}\label{Eq_ExtendedAdj}
(-)_G:s\Gp\leftrightarrows s\Alg_k:G.
\end{equation}
Unfortunately, this is not a Quillen pair.
However, one still has

\begin{lem}[{\cite{BRY22}*{Lemma 3.1}}]\label{Lemma_ExistenceOfDerivedFunctor}
The functor $(-)_G$ in (\ref{Eq_ExtendedAdj}) maps the weak equivalences between cofibrant objects in $s\Gp$ to weak equivalences in $s\Alg_k$, and hence has a total left derived functor
\begin{displaymath}
\mathbb{L}(-)_G:\Ho(s\Gp)\to\Ho(s\Alg_k).
\end{displaymath}
As a consequence, the left derived functor $\mathbb{L}(-)_G$ can be computed by $\Gamma\mapsto(Q\Gamma)_G$, where $Q\Gamma$ is a cofibrant replacement of $\Gamma$ in $s\Gp$.
\end{lem}

We shall also denote the derived functor by $\mathscr{O}(\mathrm{DRep}_G(\Gamma)):\mathrm{Ho}(s\Gp)\to\mathrm{Ho}(s\Alg_k)$ and call the corresponding (derived) scheme $\mathrm{DRep}_G(\Gamma)$ the {\bf derived representation scheme}.

\begin{rmk}
A very important implication of \Cref{Lemma_ExistenceOfDerivedFunctor} is that the homotopy group of $\mathbb{L}(\Gamma)_G$ for any simplicial group $\Gamma$ does \emph{not} depend on the choice of the cofibrant replacement $Q\Gamma\to\Gamma$.
Thus the representation homology is an algebraic invariant on the simplicial group $\Gamma$ and the group scheme $G$, and we can take approachable resolution for the sake of computation.
Later in this paper we will take a specific cofibrant replacement of the simplicial group $\Gamma$, where $\Gamma$ is given by a topological space $X$, and the resolution is constructed from a particular presentation of that space.
This replacement yields Koszul complexes that can be used for explicit computations (\Cref{Eg_RepHomOfSurfaces}).
\end{rmk}

We recall a few basic facts from simplicial homotopy theory \cite{GJ09}.

There is a Quillen equivalence
\begin{displaymath}
|-|:s\Set\rightleftarrows\Top:S
\end{displaymath}
where $S$ is the (total) singular complex and $|-|$ is the geometric realisation (\cite{GJ09}*{Theorem I.11.4.}). By definition, given any topological space $X$,
\begin{displaymath}
S_n(X):=\mathrm{Hom}_\Top(\Delta^n,X)
\end{displaymath}
where $\displaystyle\Delta^n:=\left\{(x_0,\cdots,x_n)\in\mathbb{R}^{n+1}\mid\sum_{i=0}^nx_i=1,\;x_i\geq0\right\}$ is the standard geometric simplex.

A simplicial set $X$ is called reduced if it has a single vertex, that is, $X_0=\{*\}$. The full subcategory of $s\Set$ consisting of reduced simplicial sets will be
denoted $s\Set_0$.
For a pointed topological space $(X,*)$, its Eilenberg subcomplex of $\bar{S}_*(X)$ is defined by
\begin{displaymath}
\bar{S}_n(X):=\{f:\Delta^n\to X\mid f(v_i)=* \text{ for all vertices }v_i\in\Delta^n\}.
\end{displaymath}
If $X$ is connected, the natural inclusion $\bar{S}_*(X)\hookrightarrow S_*(X)$ is a weak equivalence of simplicial sets. Further, if we restrict functor $\bar{S}$ to the category $\Top_{0,*}$ of connected pointed spaces, we get the pair of adjoint functors
\begin{equation}\label{Eq_ReducedSS}
|-|:s\Set_0\rightleftarrows\Top_{0,*}:\bar{S}.
\end{equation}
~\par

Furthermore, there is another Quillen equivalence
\begin{equation}\label{Eq_KanLoopGroup}
\mathbb{G}:s\Set_0\rightleftarrows s\Gp:\bar{W},
\end{equation}
where $\mathbb{G}$ is called the Kan loop group functor and $\bar{W}$ is the classifying simplicial complex (see {\cite{GJ09}*{Proposition V.6.3.}}). It is worthwhile to point out that by the construction, the Kan loop group $\mathbb{G}X$ of any reduced simplicial set $X$ is semi-free, hence cofibrant.

In conclusion we have a sequence of equivalences of homotopy categories
\begin{equation}\label{Eq_Simplicial}
\Ho(\Top_{0,*})\simeq\Ho(s\Set_0)\simeq\Ho(s\Gp).
\end{equation}
This means giving a simplicial group is the same as giving a pointed, connected topological group.

The famous Dold-Kan correspondence states that there is a Quillen equivalence
\begin{displaymath}
N:s\Vect_k\rightleftarrows\Com(\Vect_k):\Gamma
\end{displaymath}
when $N$ is the normalisation functor and $\Gamma$ is its right adjoint (\cite{GJ09}*{Corollary III.2.3, Corollary III.2.7.}).
Then for any simplicial vector space $V\in s\Vect_k$, we define the homotopy group $\pi_*(V)$ to be the homology group of its normalisation $N(V)$, i.e.
\begin{equation*}
\pi_*(V):=H_*(N(V)).
\end{equation*}
~\par

This leads to the definition

\begin{defn}
Given a reduced simplicial set $X$ and a group scheme $G$ over $k$, the \textbf{representation homology} $HR_*(X,G)$ of $X$ with coefficient $G$ is the homotopy group
\begin{displaymath}
\pi_*((\mathbb{G}(X))_G).
\end{displaymath}
\end{defn}

Note that a reduced simplicial set can be equivalently described by a pointed, connected topological space via the adjoint pair given in \Cref{Eq_ReducedSS}.
In particular, we reiterate that we will not distinguish a simplicial set and its geometric realization.\par
~\par

Recall that the representation homology commutes with homotopy colimit (\cite{BRY22}*{Theorem 3.2.}), and based on this important fact we have the following

\begin{eg}[See also {\cite{BRY22}*{Section 6.1.2.}}]\label{Eg_RepHomOfSurfaces}
Given a positive integer $g$, let $\Sigma_g$ be the closed orientable surface of genus $g$. Then one has a homotopy colimit diagram in $\Top_*$
\begin{displaymath}
\Sigma_g\simeq\hcolim\left[\{*\}\leftarrow S^1\xrightarrow{\alpha}\bigvee_{i=1}^{2g}S^1\right]
\end{displaymath}
where the map $\alpha$ is the attaching map of the $2$-cell onto $\bigvee_{i=1}^{2g}S^1$. In terms of the generators of fundamental groups, we could denote the map $\alpha^g$ by
\begin{align*}
\alpha^g:S^1_c&\to\bigvee_{i=1}^{g}(S^1_{a_i}\vee S^1_{b_i})\\
c&\mapsto[a_1,b_1][a_2,b_2]\cdots[a_g,b_g]
\end{align*}
where $c$ is a generator of $\pi_1(S^1)$ and $a_i,b_i$ are the generators of $\pi_1(S^1_{a_i})$ and $\pi_1(S^1_{b_i})$. Then by applying the Kan loop group construction (see \Cref{Eq_KanLoopGroup}), one get a simplicial group model for $T^2$:
\begin{displaymath}
\mathbb{G}(\Sigma_g)\simeq\hcolim\left[1\leftarrow\mathbb{F}_1\xrightarrow{\alpha^g}\mathbb{F}_{2g}\right].
\end{displaymath}
Since the representation homology commutes with homotopy colimit (\cite{BRY22}*{Theorem 3.2.}), by \Cref{Eg_RepVars} one applies the functor $\mathbb{L}(-)_G$ and has 
\begin{displaymath}
HR_*(\Sigma_g,G)\simeq\hcolim[k\xleftarrow{\epsilon}\mathscr{O}(G)\xrightarrow{\alpha_*^g}\mathscr{O}(G^{2g})]
\end{displaymath}
where the map $\epsilon:\mathscr{O}(G)\to k$ is induced by the unit map of the group $G$ and $\alpha_*^g:\mathscr{O}(G)\to\mathscr{O}(G^{2g})$ is
\begin{equation}\label{Eq_Twist}
\alpha_*^g(f)(x_1,y_1,\cdots,x_g,y_g):=f([x_1,y_1][x_2,y_2]\cdots[x_g,y_g]),
\end{equation}
or equivalently
\begin{align*}
\alpha^*_g:G^{2g}(R)&\to G(R)\\
(x_1,y_1,\cdots,x_g,y_g)&\mapsto[x_1,y_1][x_2,y_2]\cdots[x_g,y_g].
\end{align*}

In the category $s\Alg_k$, the colimit
\begin{displaymath}
\colim[k\xleftarrow{\epsilon}\mathscr{O}(G)\xrightarrow{\alpha_*^g}\mathscr{O}(G^{2g})]
\end{displaymath}
is
\begin{equation}\label{Eq_GeneralCommutingCoordinateRing}
\mathscr{O}(G)^{\otimes2g}\otimes_{\mathscr{O}(G)}k
\end{equation}
where $\mathscr{O}(G)^{\otimes2g}$ and $k$ are viewed as $\mathscr{O}(G)$-modules via $\alpha^g_*$ and $\epsilon$. Now by viewing the homotopy colimit as the derived functor of colimit,
\begin{displaymath}
\hcolim[k\xleftarrow{\epsilon}\mathscr{O}(G)\xrightarrow{\alpha_*^g}\mathscr{O}(G^{2g})]\simeq\mathscr{O}(G)^{\otimes2g}\otimes_{\mathscr{O}(G)}^{\mathbb{L}}k,
\end{displaymath}
where $-\otimes_{\mathscr{O}(G)}^{\mathbb{L}}-$ is the left derived functor of $-\otimes_{\mathscr{O}(G)}-$.
\end{eg}

Summrising \Cref{Eg_RepHomOfSurfaces}, we have

\begin{lem}\label{Lemma_DerivingCommutingVarOfHigherGenus}
Let $\Sigma_g$ be the Riemann surface of genus $g$.
\begin{displaymath}
HR_*(\Sigma_g,G)\cong H_*(\mathscr{O}(G)^{\otimes2g}\otimes_{\mathscr{O}(G)}^{\mathbb{L}}k)
\end{displaymath}
where $-\otimes_{\mathscr{O}(G)}^{\mathbb{L}}-$ is the left derived functor of $-\otimes_{\mathscr{O}(G)}-$ in $s\Alg_k$.
\end{lem}

\begin{eg}\label{Eg_DerivingCommutingVarOfHigherGenus}
One could show that $HR_0(X,G)\cong(\pi_1(X))_G$ by comparing the universal mapping properties (see also \cite{BRY22}*{(3.6)}). In fact, for any $k$-algebra $R$,
\begin{align*}
\mathrm{Hom}_{\Alg_k}((\pi_1(X))_G,R)&\cong\mathrm{Hom}_{\Gp}(\pi_1(X),G(R))\\
&\cong\mathrm{Hom}_{s\Gp}(\mathbb{G}(X),G(R))\\
&\cong\mathrm{Hom}_{s\Alg_k}((\mathbb{G}(X))_G,R)\\
&\cong\mathrm{Hom}_{\Alg_k}(H_0((\mathbb{G}(X))_G),R)
\end{align*}
where the second and the fourth equality are because $G(R)$ and $R$ are viewed as simplicial objects concentrated on $0$-th degree, i.e. all higher elements are degenerate.
In particular one has
\begin{displaymath}
HR_0(\Sigma_g,G)\cong\mathscr{O}(C_g(G))\cong\mathscr{O}(G)^{\otimes2g}\otimes_{\mathscr{O}(G)}k.
\end{displaymath}
This means the representation homology of Riemann surface of genus $g$ is a derived version of the commuting scheme of higher genus $g$.\par
To see from the definition that $\mathscr{O}(G)^{\otimes2g}\otimes_{\mathscr{O}(G)}k$ is the coordinate ring, by the universal property of colimit,
\begin{align*}
\mathrm{Hom}_{\Alg_k}(\mathscr{O}(G)^{\otimes2g}\otimes_{\mathscr{O}(G)}k,R)&\cong\{(f,h)\in\mathrm{Hom}_{\Alg_k}(\mathscr{O}(G)^{\otimes2g},R)\times\mathrm{Hom}_{\Alg_k}(k,R)\mid f\alpha^g_*=h\epsilon\}\\
&=\{(f_1,\cdots,f_{2g},h)\in G^{2g}(R)\times e(R)\mid (f_1,\cdots,f_{2g})\alpha^g_*=h\epsilon\}.
\end{align*}
By \Cref{Eq_Twist}, $(f_1,\cdots,f_{2g})\alpha^g_*=[f_1,f_2]\cdots[f_{2g-1},f_{2g}]$, and $e(R)$ consists of only the identity element, so
\begin{align*}
&\{(f_1,\cdots,f_{2g},h)\in G^{2g}(R)\times e(R)\mid (f_1,\cdots,f_{2g})\alpha^g_*=h\epsilon\}\\
=&\{(f_1,\cdots,f_{2g})\in G^{2g}(R)\mid[f_1,f_2]\cdots[f_{2g-1},f_{2g}]=1\}.
\end{align*}
\end{eg}

\begin{rmk}\label{Rmk_CDGA}
When $\mathrm{char}~k=0$, there is a model structure on $\CDGA_k$ (\cite{GS07}*{Theorem 3.6, Example 3.7.}, or explicitly \cite{GM03}*{Theorem V.3.4.}). Now there is a Quillen equivalence (\cite{Quillen69}*{Page 220-223}, or \cite{BRY22-2}*{Appendix})
\begin{displaymath}
N^*:\CDGA_k\leftrightarrows s\Alg_k:N,
\end{displaymath}
hence
\begin{align*}
HR_*(\Sigma_g,G)&:= H_*\bigg(N(\hcolim[k\xleftarrow{\epsilon}\mathscr{O}(G)\xrightarrow{\alpha_*^g}\mathscr{O}(G^{2g})])\bigg)\\
&\cong H_*\bigg(\hcolim[N(k)\xleftarrow{\epsilon}N(\mathscr{O}(G))\xrightarrow{\alpha_*^g}N(\mathscr{O}(G^{2g}))]\bigg)\\
&\cong H_*\bigg(\hcolim[k\xleftarrow{\epsilon}\mathscr{O}(G)\xrightarrow{\alpha_*^g}\mathscr{O}(G^{2g})]\bigg)
\end{align*}
where the last homotopy colimit is taken in the category $\CDGA_k$.
This means (when $\mathrm{char}~k=0$ as the assumption at the beginning) the homotopy colimit $\mathscr{O}(G)^{\otimes2g}\otimes_{\mathscr{O}(G)}^{\mathbb{L}}k$ in $s\Alg_k$ could be computed in $\CDGA_k$.
\end{rmk}

\section{Main Results}\label{Sec_Main}

The main tools we need are the following propositions:

\begin{prop}\label{Proposition_UnipotnetGp}
Let $U_n$ be the affine group scheme consisting of unipotent upper triangular matrices, and let $\Sigma_g$ be the closed orientable surface of genus $g$. Then the following are equivalent:
\begin{enumerate}
\item The representation homology $HR_i(\Sigma_g,U_n)$ vanishes in dimension greater or equal than $n$, namely
\begin{displaymath}
HR_i(\Sigma_g,U_n)=0\;\;\;\forall i\geq n.
\end{displaymath}
\item There is an isomorphism of graded algebras
\begin{equation}\label{Eq_MultiStructU}
HR_*(\Sigma_g,U_n)\cong HR_0(\Sigma_g,U_n)\otimes\mathrm{Sym}_k(T_1,\cdots,T_{n-1})
\end{equation}
where $\mathrm{Sym}_k$ is the graded symmetric algebra over $k$ and $T_i$ is of homological degree $1$.
\item The commuting scheme of higher genus $C_g(U_n)$ is a complete intersection of codimension $\frac{(n-2)(n-1)}{2}$ in $U_n\times\overset{2g}{\cdots}\times U_n$.
\end{enumerate}
\end{prop}

And similarly

\begin{prop}\label{Proposition_BorelGp}
Let $B_n$ be the affine group scheme consisting of invertible upper triangular matrices, and let $\Sigma_g$ be the closed orientable surface of genus $g$. Then the following are equivalent:
\begin{enumerate}
\item The representation homology $HR_i(\Sigma_g,B_n)$ vanishes in dimension greater than $n$, namely
\begin{displaymath}
HR_i(\Sigma_g,B_n)=0\;\;\;\forall i>n.
\end{displaymath}
\item There is an isomorphism of graded algebras
\begin{equation}\label{Eq_MultiStructB}
HR_*(\Sigma_g,B_n)\cong HR_0(\Sigma_g,B_n)\otimes_k\mathrm{Sym}_k(T_1,\cdots,T_{n})
\end{equation}
where $\mathrm{Sym}_k$ is the graded symmetric product over $k$ and $T_i$ is of homological degree $1$.
\item The commuting scheme of higher genus $C_g(B_n)$ is a complete intersection of codimension $\frac{(n-1)n}{2}$ in $B_n\times\overset{2g}{\cdots}\times B_n$.
\end{enumerate}
\end{prop}

In the propositions, we have not only the vector space structure of $HR_*(\Sigma_g,U_n)$ (resp. $HR_*(\Sigma_g,B_n)$), but also the algebra structure on it when the vanishing condition holds.

\begin{rmk}
It seems from \Cref{Eq_MultiStructU} and \Cref{Eq_MultiStructB}, one can expect that the higher representation homologies are determined by the $0$-th homology.
However, this is generally not true.
One reason is that the representation homology depends on the homotopy type of the space, while the classical representation scheme $\mathrm{Rep}_G(\pi_1(X))$ depends only on the fundamental group.
Also, it is a coincidence that we have \Cref{Eq_MultiStructU}, mainly because we take coefficients in unipotent groups.
If we take classical reductive group, things will be very different, for instance see \cite{BRY19}*{Conjecture 1.3}.
We will also provide a counterexample in \Cref{Proposition_U6NonCI}.
\end{rmk}

Now we can state our main result:

\begin{thm}\label{Thm_List}
$C(U_2),C(U_3),C(U_4),C(U_5),C(B_2),C(B_3)$ are complete intersections.
$C(U_n)$ is not a complete intersection when $n\geq6$.
\end{thm}

\begin{proof}
We use (version 1.0 of) the \verb|Macaulay2| package \verb|RepHomology| (\cite{li2024package}) and verify that the condition
\begin{displaymath}
HR_i(\Sigma_g,U_n)=0\;\;\;\forall i\geq n
\end{displaymath}
holds for $n\leq5$.
Therefore by \Cref{Proposition_UnipotnetGp}, $C(U_2),C(U_3),C(U_4),C(U_5)$ are complete intersections.

Also, the condition
\begin{displaymath}
HR_i(\Sigma_g,B_n)=0\;\;\;\forall i>n
\end{displaymath}
holds for $n\leq3$ by the package, so $C(B_2),C(B_3)$ are complete intersections by \Cref{Proposition_BorelGp}.

The code for this package can be found at
\href{https://github.com/GuanyuLee/RepresentationHomology}{https://github.com/GuanyuLee/RepresentationHomology}
and is expected to appear in a future release of \verb|Macaulay2|.

$C(U_n)$ is not a complete intersection when $n\geq6$ directly by \Cref{Proposition_U6NonCI} below.
\end{proof}

\subsection{Proofs of the main results}

We will break the proof into several lemmas. By \Cref{Rmk_CDGA}, we can proceed our computation in the category $\CDGA_k$.

\begin{lem}\label{Eg_RegularKernelInCDGA}
Suppose that the map $\epsilon:\mathscr{O}(G)\to k$, which is induced by taking the identity in the group $G$, has a kernel generated by a regular sequence $\{y_1,\cdots,y_r\}$.
Then
\begin{displaymath}
HR_*(\Sigma_g,G)\cong H_*(\mathrm{Kos}_{\mathscr{O}(G^{2g})}(\alpha_*(y_1),\cdots,\alpha_*(y_r)))
\end{displaymath}
where $\alpha_*$ is defined in \Cref{Eq_Twist}.
\end{lem}
\begin{proof}
Since $\ker\epsilon$ is generated by a regular sequence $\{y_1,\cdots,y_r\}$, the canonical map
\begin{align*}
\mathrm{Kos}^{\mathscr{O}(G)}_*(y_1,\cdots,y_r)\to k
\end{align*}
gives a cofibrant replacement of $k$ over $\mathscr{O}(G)$ in $\CDGA_k$ by \cite{Eisenbud95}*{Corollary 17.5}. It is easy to see by the definition of Koszul complex that
\begin{displaymath}
\mathscr{O}(G)^{\otimes2g}\otimes_{\mathscr{O}(G)}\mathrm{Kos}_{\mathscr{O}(G)}(y_1,\cdots,y_r)\cong\mathrm{Kos}^{\mathscr{O}(G^{2g})}_*(\alpha_*(y_1),\cdots,\alpha_*(y_r))
\end{displaymath}
where $\alpha_*$ is described as in \Cref{Eq_Twist}.
Thus the representation homology $HR_*(\Sigma_g,G)$ can be computed by $H_*(\mathrm{Kos}_{\mathscr{O}(G^{2g})}(\alpha_*(y_1),\cdots,\alpha_*(y_r)))$ by \Cref{Lemma_DerivingCommutingVarOfHigherGenus}.
\end{proof}

In our cases, the complexes can be written explicitly.

\begin{lem}\label{Lemma_SpecificComputations}
The representation homology $HR_*(\Sigma_g,U_n)$ for $G=U_n$ can be computed by
\begin{displaymath}
HR_*(\Sigma_g,U_n)\cong H_*(\mathrm{Kos}_*^{\mathscr{O}(U_n^{2g})}(\{f_{i,j}\}_{1\leq i<j\leq n}))
\end{displaymath}
where $f_{i,j}$ is the $(i,j)$-entry of the matrix
\begin{equation}\label{Eq_UnKoszulMatrix}
f:=[X_1,Y_1]\cdots[X_g,Y_g]-I,
\end{equation}
and $X_t,Y_t$ are the coordinate matrices of the $2t-1,2t$-th copy of $U_n$ in $U_n^{2g}$, respectively (so the coordinate ring of $U_n^{2g}$ is $k[x_{t,i,j},y_{t,i,j}]_{1\leq i<j\leq n,1\leq t\leq g}$).
\end{lem}
\begin{proof}
By \Cref{Lemma_GroupProperties}, $\{x_{i,j}\}_{1\leq i<j\leq n}$ forms a regular sequence in $\mathscr{O}(U_n)$ which generates the kernel of $\epsilon:\mathscr{O}(U_n)\to k$. Thus by \Cref{Eg_RegularKernelInCDGA},
\begin{displaymath}
HR_*(\Sigma_g,U_n)\cong H_*(\mathrm{Kos}_*^{\mathscr{O}(U_n^{2g})}(\{\alpha_*^g(x_{i,j})\}_{1\leq i<j\leq n})).
\end{displaymath}
We then need to specify the elements $\alpha_*^g(x_{i,j})$ in $\mathrm{Kos}_*^{\mathscr{O}(U_n^{2g})}$.

Suppose that the $2g$-tuple of matrices $(X_1,Y_1,X_2,Y_2,\cdots,X_g,Y_g)$ gives the coordinate matrices of the copies of $U_n$ in $U_n^{2g}=(U_n^2)^g$, each of which has the same shape as \Cref{Eq_UnMatrixGenerator}. With these settings, the coordinate ring
\begin{displaymath}
\mathscr{O}(U_n^{2g})=k[X_1,Y_1,X_2,Y_2,\cdots,X_g,Y_g]=k[x_{t,i,j},y_{t,i,j}]_{1\leq i<j\leq n,1\leq t\leq g}.
\end{displaymath}

Given arbitrary points $(x_1,y_1,x_2,y_2,\cdots,x_g,y_g)\in U_n^{2g}(R)=(U_n^2)^g(R)$, view $X_i,Y_i$ (resp. $X$) matrices of functions on $U^{2g}(R)$ (resp. $U(R)$), then
\begin{align*}
\alpha_*^g(X)(x_1,y_1,\cdots,x_g,y_g)&=X([x_1,y_1][x_2,y_2]\cdots[x_g,y_g])\\
=&[X_1(x_1),Y_1(y_1)]\cdots[X_g(x_g),Y_g(y_g)]\\
=&\bigg([X_1,Y_1]\cdots[X_g,Y_g]\bigg)\bigg([x_1,y_1][x_2,y_2]\cdots[x_g,y_g]\bigg)
\end{align*}
where the first equation is by definition \Cref{Eq_Twist} and the third equation is because all $X_i$'s and $Y_i$'s are matrix generators.

This implies that $\alpha_*^g(x_{i,j})$ is the $(i,j)$-entry of the matrix
\begin{displaymath}
f:=[X_1,Y_1]\cdots[X_g,Y_g]-I.
\end{displaymath}
Thus
\begin{displaymath}
HR_*(\Sigma_g,U_n)\cong H_*(\mathrm{Kos}_{\mathscr{O}(U_n^{2g})}(\{f_{i,j}\}_{1\leq i<j\leq n}))
\end{displaymath}
where $f_{i,j}$ is the $(i,j)$-entry in the matrix $[X_1,Y_1]\cdots[X_g,Y_g]-I$ defined in \Cref{Eq_UnKoszulMatrix}.
\end{proof}

Similarly for $B_n$, it has a matrix generator \Cref{Eq_BnMatrixGenerator} and also by \Cref{Lemma_GroupProperties}, $\{x_{i,j}-\delta_j^i\}_{1\leq i\leq j\leq n}$ forms a regular sequence in $\mathscr{O}(B_n)$ which generates the kernel of $\epsilon:\mathscr{O}(B_n)\to k$. Therefore we can apply \Cref{Eg_RegularKernelInCDGA} in the same manner, yielding

\begin{lem}
The representation homology $HR_*(\Sigma_g,B_n)$ for $G=B_n$ can be computed by
\begin{displaymath}
HR_*(\Sigma_g,B_n)\cong H_*(\mathrm{Kos}_*^{\mathscr{O}(B_n^{2g})}(\{f_{i,j}\}_{1\leq i\leq j\leq n}))
\end{displaymath}
where $f_{i,j}$ is the $(i,j)$-entry of the matrix
\begin{equation}\label{Eq_BnKoszulMatrix}
f:=[X_1,Y_1]\cdots[X_g,Y_g]-I,
\end{equation}
and $X_t,Y_t$ are the coordinate matrices of the $2t-1,2t$-th copy of $B_n$ in $B_n^{2g}$, respectively.
\end{lem}

We are ready to prove the main propositions.

\begin{proof}[Proof of \Cref{Proposition_UnipotnetGp}]
Let $X_t,Y_t$ be the coordinate matrices of the $2t-1,2t$-th copy of $U_n$ in $U_n^{2g}$ respectively, and let
\begin{displaymath}
f=[X_1,Y_1]\cdots[X_g,Y_g]-I
\end{displaymath}
as in \Cref{Eq_UnKoszulMatrix}. We first notice that by \Cref{Lemma_GroupProperties}, the upper part of $f$, $\{f_{i,j}\}_{1\leq i<j\leq n}$, consists of homogeneous elements in $\mathscr{O}(U_n^{2g})=k[X_1,Y_1,X_2,Y_2,\cdots,X_g,Y_g]=k[x_{t,i,j},y_{t,i,j}]_{1\leq i<j\leq n,1\leq t\leq g}$, and by \Cref{Lemma_SpecificComputations},
\begin{displaymath}
HR_*(\Sigma_g,U_n)\cong H_*(\mathrm{Kos}_*^{\mathscr{O}(U_n^{2g})}(\{f_{i,j}\}_{1\leq i<j\leq n})).
\end{displaymath}

One notice that the sub-diagonal elements $f_{i,i+1}$ are $0$ for all $1\leq i<n$, since the sub-diagonal elements of $[X_t,Y_t]$ are $0$. Therefore, among all the elements $\{f_{i,j}\}_{1\leq i<j\leq n}$, there are exactly $n-1$ terms which are $0$, so there is an isomorphism of complexes by \Cref{Lemma_KoszulZeroSubseq}
\begin{equation}\label{Eq_UnSimpleStruct}
H_*(\mathrm{Kos}_*(\{f_{i,j}\}_{1\leq i<j\leq n}))\cong H_*(\mathrm{Kos}_*(\{f_{i,j}\}_{1\leq i<j\leq n,\,i+1<j}))\otimes_k\mathrm{Sym}_k(T_1,\cdots,T_{n-1})
\end{equation}
where each $T_j$ is of degree $1$.

By \cite{Eisenbud95}*{Corollary 17.5} and \Cref{Eq_UnSimpleStruct},
\begin{equation}\label{Eq_ExplicitCoordinateRing}
H_0(\mathrm{Kos}_*(\{f_{i,j}\}_{1\leq i<j\leq n}))=\frac{\mathscr{O}(U_n^{2g})}{(\{f_{i,j}\}_{1\leq i<j\leq n})}=\frac{\mathscr{O}(U_n^{2g})}{(\{f_{i,j}\}_{1\leq i<j\leq n,\,i+1<j})}.
\end{equation}
And combining \Cref{Eg_DerivingCommutingVarOfHigherGenus} and \Cref{Eq_ExplicitCoordinateRing},
\begin{equation*}
\mathscr{O}(C_g(U_n))=H_0(\mathrm{Kos}_*(\{f_{i,j}\}_{1\leq i<j\leq n}))=\frac{\mathscr{O}(U_n^{2g})}{(\{f_{i,j}\}_{1\leq i<j\leq n,\,i+1<j})}.
\end{equation*}

\begin{enumerate}

\item[(1)$\Rightarrow$(2)]
If the vanishing results hold, then
\begin{displaymath}
H_*(\mathrm{Kos}_*(\{f_{i,j}\}_{1\leq i<j\leq n,\,i+1<j}))\otimes\mathrm{Sym}_k(T_1,\cdots,T_{n-1})
\end{displaymath}
has at most $n-1$ nontrivial levels. However the part $\mathrm{Sym}_k(T_1,\cdots,T_{n-1})$ already provides $n-1$ nontrivial levels, hence $H_*(\mathrm{Kos}_*(\{f_{i,j}\}_{1\leq i<j\leq n,\,i+1<j}))$ is concentrated at degree $0$.

\item[(2)$\Rightarrow$(3)]
By the assumption and \Cref{Eq_UnSimpleStruct},
\begin{displaymath}
H_*(\mathrm{Kos}_*(\{f_{i,j}\}_{1\leq i<j\leq n,\,i+1<j}))\cong H_0(\mathrm{Kos}_*(\{f_{i,j}\}_{1\leq i<j\leq n,\,i+1<j})).
\end{displaymath}
Hence by \Cref{Thm_CA}, $\{f_{i,j}\}_{1\leq i<j\leq n,\,i+1<j}$ is a regular sequence. Then by \Cref{Cor_CI}, $C_g(U_n)$ is a complete intersection.

\item[(3)$\Rightarrow$(1)]
Bear in mind that \Cref{Eq_ExplicitCoordinateRing} gives the coordinate ring, by \Cref{Cor_CI} and \Cref{Thm_CA},
\begin{displaymath}
H_*(\mathrm{Kos}_*(\{f_{i,j}\}_{1\leq i<j\leq n,\,i+1<j}))\cong H_0(\mathrm{Kos}_*(\{f_{i,j}\}_{1\leq i<j\leq n,\,i+1<j})).
\end{displaymath}
From \Cref{Eq_UnSimpleStruct}, the vanishing result is proven.

\end{enumerate}
\end{proof}

\begin{rmk}
The key step of the proof is the observation that the subdiagonal elements $\{f_{i,i+1}\}_{1\leq i<n}$ are $0$ in \Cref{Eq_UnKoszulMatrix}.

For $B_n$, diagonal elements $\{f_{i,i}\}_{1\leq i\leq n}$ in \Cref{Eq_BnKoszulMatrix} are also $0$. Hence there is a similar formula
\begin{displaymath}
H_*(\mathrm{Kos}_*(\{f_{i,j}\}_{1\leq i\leq j\leq n}))\cong H_*(\mathrm{Kos}_*(\{f_{i,j}\}_{1\leq i<j\leq n}))\otimes_k\mathrm{Sym}_k(T_1,\cdots,T_{n})
\end{displaymath}
and all other tools could be applied in the same way by \Cref{Lemma_GroupProperties}. The proof of \Cref{Proposition_BorelGp} will be almost exactly the same.
\end{rmk}

\subsection{A counterexample}

\begin{prop}\label{Proposition_U6NonCI}
$C(U_6)$ is not a complete intersection of codimension $10$ in $U_6\times U_6$.
\end{prop}

\begin{proof}
Consider the generic $6\times6$ unipotent matrices
\begin{align*}
X=\begin{pmatrix}1&x_{1,2}&x_{1,3}&x_{1,4}&x_{1,5}&x_{1,6}\\
&1&x_{2,3}&x_{2,4}&x_{2,5}&x_{2,6}\\
&&1&x_{3,4}&x_{3,5}&x_{3,6}\\
&&&1&x_{4,5}&x_{4,6}\\ &&&&1&x_{5,6}\\ &&&&&1
\end{pmatrix},\;\;
&
Y=\begin{pmatrix}1&y_{1,2}&y_{1,3}&y_{1,4}&y_{1,5}&y_{1,6}\\
&1&y_{2,3}&y_{2,4}&y_{2,5}&y_{2,6}\\
&&1&y_{3,4}&y_{3,5}&y_{3,6}\\
&&&1&y_{4,5}&y_{4,6}\\ &&&&1&y_{5,6}\\ &&&&&1
\end{pmatrix},
\end{align*}
and let
\begin{displaymath}
f=\begin{pmatrix}0&f_{1,2}&\textcolor{red}{f_{1,3}}&\textcolor{red}{f_{1,4}}&f_{1,5}&f_{1,6}\\
&0&f_{2,3}&\textcolor{red}{f_{2,4}}&\textcolor{red}{f_{2,5}}&f_{2,6}\\
&&0&f_{3,4}&\textcolor{red}{f_{3,5}}&\textcolor{red}{f_{3,6}}\\
&&&0&f_{4,5}&\textcolor{red}{f_{4,6}}\\ &&&&0&f_{5,6}\\ &&&&&0
\end{pmatrix}:=XYX^{-1}Y^{-1}-I.
\end{displaymath}
Notice that we have $f_{1,2}=f_{2,3}=f_{3,4}=f_{4,5}=f_{5,6}=0$ on the subdiagonal, and $C(U_6)$ has a presentation (as in \Cref{Eq_ExplicitCoordinateRing})
\begin{equation}
C(U_6)=\Spec\frac{k[x_{i,j},y_{i,j}]}{(\{f_{i,j}\}_{1\leq i<j\leq 6})}=\Spec\frac{k[x_{i,j},y_{i,j}]}{(\{f_{i,j}\}_{1\leq i<j\leq 6,\,i+1<j})}.
\end{equation}

Suppose that $C(U_6)$ is a complete intersection.
By \Cref{Cor_CI}, the set of generators $\displaystyle\{f_{i,j}\}_{0\leq i+1<j\leq n}$ of $C(U_6)$ forms a regular sequence.
Thus by \Cref{Cor_Subseq}, its subsequence $\displaystyle\{f_{1,3},f_{1,4},f_{2,4},f_{2,5},f_{3,5},f_{3,6},f_{4,6}\}$ is again a regular sequence, so
\begin{displaymath}
\mathrm{codim}~(f_{1,3},f_{1,4},f_{2,4},f_{2,5},f_{3,5},f_{3,6},f_{4,6})=7.
\end{displaymath}

Consider the natural map
\begin{displaymath}
q:k[X,Y]\to k[X,Y]/(x_{2,3},x_{4,5},y_{2,3},y_{4,5}),
\end{displaymath}
and let the matrices
\begin{align*}
\bar{X}=\begin{pmatrix}1&x_{1,2}&x_{1,3}&x_{1,4}&x_{1,5}&x_{1,6}\\
&1&0&x_{2,4}&x_{2,5}&x_{2,6}\\
&&1&x_{3,4}&x_{3,5}&x_{3,6}\\
&&&1&0&x_{4,6}\\ &&&&1&x_{5,6}\\ &&&&&1
\end{pmatrix},\;\;
&
\bar{Y}=\begin{pmatrix}1&y_{1,2}&y_{1,3}&y_{1,4}&y_{1,5}&y_{1,6}\\
&1&0&y_{2,4}&y_{2,5}&y_{2,6}\\
&&1&y_{3,4}&y_{3,5}&y_{3,6}\\
&&&1&0&y_{4,6}\\ &&&&1&y_{5,6}\\ &&&&&1
\end{pmatrix},
\end{align*}
be the images of $X,Y$ under the quotient map.
One then has
\begin{equation}\label{Eq_Issue}
\bar{X}\bar{Y}\bar{X}^{-1}\bar{Y}^{-1}=
\begin{pmatrix}1&0&\color{red}0\color{black}&\color{red}\substack{x_{1,2}y_{2,4}+x_{1,3}y_{3,4}\\ -x_{3,4}y_{1,3}-x_{2,4}y_{1,2}}\color{black}&*&*\\
&1&0&\color{red}0\color{black}&\color{red}0\color{black}&*\\
&&1&0&\color{red}0\color{black}&\color{red}\substack{x_{3,4}y_{4,6}+x_{3,5}y_{5,6}\\ -x_{5,6}y_{3,5}-x_{4,6}y_{3,4}}\color{black}\\
&&&1&0&\color{red}0\color{black}\\ &&&&1&0\\ &&&&&1\end{pmatrix}.
\end{equation}
\Cref{Eq_Issue} implies the image of $(f_{1,3},f_{1,4},f_{2,4},f_{2,5},f_{3,5},f_{3,6},f_{4,6})$ under $q$ is
\begin{displaymath}
(x_{1,2}y_{2,4}+x_{1,3}y_{3,4}-x_{3,4}y_{1,3}-x_{2,4}y_{1,2})+(x_{3,4}y_{4,6}+x_{3,5}y_{5,6}-x_{5,6}y_{3,5}-x_{4,6}y_{3,4}).
\end{displaymath}

Thus the ideal generated by $f_{1,3},f_{1,4},f_{2,4},f_{2,5},f_{3,5},f_{3,6},f_{4,6}$ is contained in the ideal
\begin{displaymath}
(x_{2,3},x_{4,5},y_{2,3},y_{4,5})+(x_{1,2}y_{2,4}+x_{1,3}y_{3,4}-x_{3,4}y_{1,3}-x_{2,4}y_{1,2})+(x_{3,4}y_{4,6}+x_{3,5}y_{5,6}-x_{5,6}y_{3,5}-x_{4,6}y_{3,4}).
\end{displaymath}
This ideal clearly has codimension $\leq6$ (By Krull's principal ideal theorem \cite{Eisenbud95}*{Theorem 10.2}, because the ideal is generated by $6$ elements), hence we arrive at a contradiction.

Notice that we actually show that the codimension of $C(U_6)$ in $U_6\times U_6$ is strictly less than $10$, and in $U_6\times U_6$ \emph{there are at least $10$ equation required to compare entries}, hence
$C(U_6)$ is not a complete intersection.
\end{proof}

\begin{rmk}
The proof of \Cref{Proposition_U6NonCI} also shows that $U_n$ for $n\geq6$ are not complete intersection of the correct dimension, since for bigger $n$, the matrix generator (see \Cref{Eg_MatrixGenerator}) of $U_n$ contains a $6\times6$ matrix as its submatrix.
\end{rmk}

\begin{rmk}
The argument of \Cref{Proposition_U6NonCI} does not work when $n=5$, since if the matrix is smaller by $1$, the last column of \Cref{Eq_Issue} disappears, and the ideal generated by the subsequence $\{f_{1,3},f_{1,4},f_{2,4},f_{2,5},f_{3,5}\}$ is of codimension $5$, which is contained in the ideal $(x_{2,3},x_{4,5},y_{2,3},y_{4,5})+(x_{1,2}y_{2,4}+x_{1,3}y_{3,4}-x_{3,4}y_{1,3}-x_{2,4}y_{1,2})$ of codimension $5$.
\end{rmk}

\end{document}